\documentclass{amsart}
\usepackage{amsmath,amsfonts, amssymb,bm}
\usepackage[]{hyperref}
\usepackage{xcolor}
\usepackage{ulem,xparse}
\usepackage{setspace}  

\makeatletter
\newcommand{\psum}{%
  \DOTSB\mathop{%
    \smash{\sideset{}{'}\sum}%
    \vphantom{\sum}%
  }\slimits@
}
\makeatother

\newcommand{\FF}{\mathbb{F}}
\newcommand{\TT}{\mathbb{T}}
\newcommand{\LL}{\mathbb{L}}
\newcommand{\CC}{\mathbb{C}}

\newcommand{\bA}{\mathbb{A}}
\newcommand{\KK}{\mathbb{K}}
\newcommand{\pfrak}{\mathfrak{p}}
\newcommand{\qfrak}{\mathfrak{q}}

\newcommand{\ffrak}{\mathfrak{f}}
\newcommand{\Gal}{\operatorname{Gal}}
\newcommand{\pitilde}{\widetilde{\pi}}

\newcommand{\GL}{\operatorname{GL}}
\newcommand{\Frac}{\operatorname{Frac}}
\newcommand{\ts}[1]{\mathcal{D}_{#1}}  
\newcommand{\ps}[1]{[\![#1]\!]}  
\newcommand{\ls}[1]{(\!(#1)\!)}
\newcommand{\Fq}{\FF_{\! q}}
\newcommand{\Fqac}{\Fq^{\rm ac}}
\newcommand{\CI}{\CC_\infty}
\newcommand{\hd}[2]{\partial_t^{(#1)}\!\!\left(#2\right)}
\newcommand{\hde}[1]{\partial_t^{(#1)}}  
 
\newcommand{\sep}{{\rm sep}}
\newcommand{\id}{{\rm id}}

\newcommand{\Fper}{F^{\rm perf}}
\newcommand{\Fac}{F^{\rm ac}}
\newcommand{\Fsep}{F^{\rm sep}}
\newcommand{\dumot}{M_\phi}

\newenvironment{smatrix}{\left(\begin{smallmatrix}}{\end{smallmatrix}\right)} 

\theoremstyle{plain}
\newtheorem{thm}{Theorem}[section]
\newtheorem*{thm*}{Theorem}
\newtheorem{cor}[thm]{Corollary}
\newtheorem{lem}[thm]{Lemma}
\newtheorem{prop}[thm]{Proposition}

\newtheorem*{thmA}{Theorem A}
\newtheorem*{thmB}{Theorem B}

\theoremstyle{definition}
\newtheorem{defn}[thm]{Definition}

\newtheorem{rem}[thm]{Remark}

\title[Taylor coefficients of Anderson generating functions]{Taylor coefficients of Anderson generating functions and Drinfeld torsion extensions}
\author{A. Maurischat}
\address{FH Aachen University of Applied Sciences \& RWTH Aachen University,
Germany}
\email{maurischat@fh-aachen.de}

\author{R. Perkins}
\address{University of California, San Diego}
\email{rudolph.perkins@gmail.com}
\date{March 24th, 2021}

\begin{document}

\begin{abstract}
We generalize our work on Carlitz prime power torsion extension to
torsion extensions of Drinfeld modules of arbitrary rank.\\
As in the Carlitz case, we give a description of these extensions in terms of evaluations of  Anderson generating functions and their hyperderivatives at roots of unity. 
We also give a direct proof that the image of the Galois representation attached to the  $\pfrak$-adic Tate module lies in the $\pfrak$-adic points of the motivic Galois group. This is a generalization of the corresponding result of Chang and Papanikolas for the $t$-adic case.
\end{abstract}

\maketitle

\setcounter{tocdepth}{1}

\tableofcontents

\section{Introduction}\label{sec:intro}

In function field arithmetic, interesting Galois extensions arise by adjoining torsion points of Drinfeld modules over a given function field $F$ to that function field. The first part of this article is to show that the prime power torsion points can be obtained in a unified manner by evaluating certain functions at roots of the different primes. In the second part, we use this result to prove a direct link between the Galois representations attached to the Drinfeld module and the motivic Galois group.
To become more precise, we introduce some notation.

Let $\Fq$ be the finite field with $q$ elements, and let $\theta$ an indeterminate over this field. We are interested in Drinfeld $\Fq[\theta]$-modules $\phi$ of rank $r$ defined over some finite extension $F$ of $K := \Fq(\theta)$. Let $\CI$ be the completion of the algebraic closure of $K_\infty := \Fq\ls{1/\theta}$ equipped with the canonical extension $|\cdot|$ of the absolute value for which $K_\infty$ is complete and normalized so that $|\theta| = q$. We consider $F$ being embedded in $\CI$.

We fix a monic irreducible polynomial $\pfrak \in A=\Fq[\theta]$ (called a \textit{prime} of $A$), and its roots $\zeta=\zeta_1, \zeta_2,\ldots, \zeta_d\in \Fqac$ where $d=\deg(\pfrak)$. We will focus on the following {\it prime power torsion extensions} of $F$,
\[ F_n:=F(\phi[\pfrak^{n+1}]); \quad n \geq 0,\]
arising by adjoining the  $\pfrak^{n+1}$-torsion of $\phi$ to $F$.

In \cite{MP}, we considered the case of $\phi=C$ being the Carlitz module. There, we reproved a theorem of Angl\`es-Pellarin (see \cite[Theorem 3.3]{APinv}) using different methods that the torsion extension $F(\zeta,C[\pfrak^{n+1}])$ of $F(\zeta)$ is generated by the evaluations of the $n$-th hyperderivative of the Anderson-Thakur function at the roots of $\pfrak$. Our method, however, revealed more information and enabled us to provide an integral basis of the extension using monomials in these values (cf.~\cite[Thm.~3.15]{MP}).

In this article, we generalize our method to Drinfeld modules of arbitrary rank.
Let $z_1,\ldots, z_r$ be an $A$-basis of the period lattice $\Lambda_\phi$ of $\phi$, and $\omega_1,\ldots,\omega_r$ the associated 
Anderson generating functions. We show the following (see Cor.~\ref{cor:generation-by-omega-n-zeta-k} and Thm.~\ref{Fqbasisthm}).

\begin{thmA} 
For all $n\geq 0$,
\[  F(\zeta, \omega_j^{(n)}(\zeta_k)| 1\leq j\leq r;\, 1\leq k\leq d) = F_n(\zeta)
= F(\zeta,\phi[\pfrak^{n+1}]). \]

Furthermore, define $c_{j,(m),l}\in\CI$ by $\omega_j^{(m)}(\zeta_k)=\sum_{l=0}^{d-1} c_{j,(m),l}\zeta_k^l$ for all $k=1,\ldots,d$.
Then the set of coefficients 
\[ \{ c_{j,(m),l} \mid 1\leq j\leq r; 0\leq m\leq n; 0\leq l\leq d-1 \} \]
forms an $\Fq$-basis of $\phi[\pfrak^{n+1}]$.
\end{thmA}

Here $\omega_j^{(n)}$ is the $n$-th hyperderivative of $\omega_j$ (see Section \ref{sec:generalities}).
We also consider the $\pfrak$-adic Tate module 
$T_\pfrak(\phi)=\lim\limits_{\leftarrow} \phi[\pfrak^{n+1}]$
which is a free rank $r$ module over $A_\pfrak$, the completion of $A$ at $\pfrak$. The Tate module is equipped with an action of the absolute Galois group
$\Gal(F^{\sep}/F)$ which after a choice of a basis of $T_\pfrak(\phi)$ defines a representation
\[\varphi_\pfrak:\Gal(F^{\sep}/F) \rightarrow \GL_r(A_\pfrak). \]
The presentation of the torsion extensions via the $\omega_j^{(n)}(\zeta_k)$, and the fact that the Anderson generating functions $\omega_j$ and their twists are also closely related to the entries of a rigid analytic trivialization $\Psi$ for the associated dual $t$-motive $\dumot$ of $\phi$, are the basis to a direct link between the image of the Galois representation $\varphi_\pfrak$ and the Galois group $\Gamma_{\Psi}$ associated to the rigid analytic trivialization.

\begin{thmB} (Theorem~\ref{thm:image-of-gal-rep-in-motivic-gal-grp})\\
Let $K_\pfrak$ be the field of fractions of $A_\pfrak$. Then
 \[\varphi_\pfrak(\Gal(F^{\rm sep}/F))\subseteq \Gamma_{\Psi}(A_\pfrak):=\GL_r(A_\pfrak)\cap \Gamma_{\Psi}(K_\pfrak).\]
\end{thmB}

The statement itself is already known (see e.g.~\cite[Section 6]{HJ}), but our proof provides an explicit description
of the representation $\Gal(F^{\rm sep}/F)\to \Gamma_{\Psi}(A_\pfrak)$. Our approach to prove the theorem is a generalization of Chang-Papanikolas' way of proving the case $\pfrak=\theta$ (see \cite[Theorem 3.5.1]{cc-mp:aipldm}).
Roughly speaking, $\sigma\in \Gal(F^{\rm sep}/F)$ respects the $F(t)$-relations among the entries of $\Psi$, and hence $\sigma(\Psi)$ also satisfies these relations.
Therefore the matrix $\Psi^{-1}\cdot \sigma(\Psi)$ which is the image of
$\sigma$ via the Galois representation is in the motivic group.
Compared to Chang-Papanikolas, the main difficulty is the preparation for obtaining it. 
We have to introduce Taylor expansions of the $\omega_j$ at $\zeta_k$, and to verify that the Galois action on these Taylor expansions is given by the representation $\varphi_\pfrak$.
These two points were immediate in the $\theta$-adic case, as the Anderson
generating functions are already given as generating series for the $\theta$-power torsion.

The article is structured as follows. In Section \ref{sec:generalities}, we provide the general definitions and notation used in the sequel. In Section \ref{sec:coefficients-are-torsion}, we consider the evaluations of the Anderson generating functions at roots of unity, and in particular we prove Theorem A. Section \ref{sec:galois-action} is dedicated to the Galois action of these values which lays the foundation to the proof of Theorem B given in Section \ref{sec:link-to-motivic-group}.
Finally, in Section \ref{sec:modular-forms}, we start some considerations on modularity when the lattice varies. We intend to work out these considerations in future work.

\section{Generalities}\label{sec:generalities}
\subsection{Basic Notation}
We will use the following notation which is almost the same as in \cite{MP}.

\begin{spacing}{1.2}
\begin{tabbing} 
\hspace*{2.9cm} \=  \kill
$\Fq$\> finite field with $q$ elements and characteristic $p$,\\
$\Fqac$\> algebraic closure of $\Fq$,\\
$A=\Fq[\theta]$\>  polynomial ring in the indeterminate $\theta$,\\
$K=\Fq(\theta)$\>  field of fractions of $A$,\\
$|\cdot|$\> norm on $K$ (and $K_\infty$ and $\CI$) given by $|a|=q^{\deg(a)}$ for $a\in A$,\\
$K_\infty= \Fq\ls{\tfrac{1}{\theta}}$\>  the completion of $K$ with respect to $|\cdot|$,\\
$\CI$\> the completion of the algebraic closure of $K_\infty$,\\
$\tau:\CI\to \CI$\> $q$-Frobenius homomorphism given by $\tau(x)=x^q$,\\ 
$\pfrak\in A$\>  monic irreducible polynomial,\\
$d=\deg(\pfrak)$\>  degree of $\pfrak$,\\
$\zeta\in \Fqac \subset \CI$\> arbitrary root of $\pfrak$,\\
$\zeta_1,\ldots, \zeta_d\in \Fqac$\>  all roots of $\pfrak$  with $\zeta_{i+1}= \zeta_i^q$ for $1 \leq i \leq d-1$, and $\zeta_1= \zeta_d^q$,\\
$A_\pfrak$ \> $\pfrak$-adic completion of $A$,\\
$K_\pfrak$ \>  $\pfrak$-adic completion of $K$ equal to the field of fractions of $A_\pfrak$,\\
$F\subset \CI$\> some field containing $K$ (usually a finite extension of $K$),\\
$\Fsep,\Fac\subset \CI$\> the separable algebraic resp.~the algebraic closure of $F$ in $\CI$,\\
$\Fper\!=\!F^{\frac{1}{p^\infty}}\!\subset\! \CI$\> the perfect closure of $F$ in $\CI$,\\
$\TT$\> Tate algebra over $\CI$ in one indeterminate $t$, i.e., \\
\> $\TT  = \left\{ \sum_{i \geq 0} c_i t^i \in \CI\ps{t} : c_i \rightarrow 0 \text{ as } i \rightarrow \infty \right\}$,\\
\> power series converging on the unit disc,\\
$ \chi_t:A \to \TT$\> $\Fq$-algebra homomorphism  determined by $\chi_t(\theta) = t$,\\
$a(t) \in \TT$\> further notation for the image of $a \in A$ under $\chi_t$,\\
$h(z)=h|_{t=z}$ \> evaluation of $h\in \TT$ at $z\in \CI$ with $|z|\leq 1$.
\end{tabbing}
\end{spacing}
Furthermore, we denote by $R^{n \times m}$ the space of $n\times m$ matrices with coefficients in a ring $R$. If we have an action of a group, etc. on $R$ we shall write $\sigma(T) = (\sigma T_{ij})$ for $\sigma$ in the group and $T = (T_{ij}) \in R^{n \times m}$.

\pagebreak

\subsection{Operations on the Tate algebra} 
In this subsection, we introduce twists and hyperdifferential operators on power series in the Tate algebra $\TT$, as well as some properties of these.\footnote{The twists and hyperdifferential operators can also be defined for arbitrary power series, but we will not need this.
}

For $h = \sum c_i t^i \in \TT$, we define the {\it $j$-th Anderson twist} of $h$,
\[h^{\tau^j} := \sum \tau^j(c_i) t^i= \sum c_i^{q^j} t^i,\]
and extend these $\CI$-linearly to the (non-commutative) twisted polynomial ring $\CI\{\tau\}$. We write
\[ h^{\ffrak} := \sum_{i \geq 0} \ffrak(c_i) t^i\]
for the action of $\ffrak \in \CI\{\tau\}$ on $h \in \TT$. \footnote{
Although, the notation with a superscript might suggest that it is a right action, we indeed have a left action of $\CI\{\tau\}$. However, this will never cause a problem in this article.}

We also employ the \textit{hyperdifferential operators} with respect to $t$, i.e.~the sequence of $\CC_\infty$-linear maps $(\hde{n})_{n\geq 0}$ given by
\[    \hd{n}{\sum_{i=0}^\infty c_it^i} =\sum_{i=0}^\infty \binom{i}{n} c_it^{i-n}=\sum_{i=n}^\infty \binom{i}{n} c_it^{i-n}, \]
where $\binom{i}{n}\in \FF_{\! p}\subset \Fq$ is the binomial coefficient modulo $p$.
We will shortly write $h^{(n)}$ for the \textit{hyperderivative} $\hd{n}{h}$, and note that the hyperdifferential operators satisfy the Leibniz rule
\[ (hg)^{(n)} = \sum_{j = 0}^n h^{(j)} g^{(n-j)},\]
among other calculus rules.
Obviously, the $\CI\{\tau\}$-action and the hyperdifferential operators commute with each other.

For $a \in A$, we may abuse notation by writing $a^{(n)}$ for $\chi_t(a)^{(n)}=a(t)^{(n)}$. 

As all elements in $\Fqac$ have norm equal to $1$, we can evaluate power series in the Tate algebra at the $\zeta_k\in \Fqac$.
As shown in \cite[Lemma 2.1]{MP}, this evaluation and the twisting interact in the following very nice manner. 

\begin{lem}[{\bf Evaluation at roots of unity and twisting}]  \label{twistevallem}  \
\begin{enumerate}
\item For $h \in \TT$, there are unique $f_0,\dots, f_{d-1}\in \CI$ such that
for every $k = 1,\dots,d$, one has
\begin{equation} h|_{t = \zeta_k} = \sum_{l = 0}^{d-1} f_l \zeta_k^l.  \label{TTrtof1eval} \end{equation}
\item For all $\ffrak \in \CI\{\tau\}$, with $h$ written as in \eqref{TTrtof1eval} above, we have for every $k = 1,\dots,d$,
\[ h^{\ffrak}|_{t = \zeta_k} = \sum_{l = 0}^{d-1} \ffrak(f_l) \zeta_k^l. \]
\end{enumerate}
\end{lem}

\subsection{Taylor series}

From the properties of the hyperdifferential operators, it follows that the map
\[ \ts{} : \TT \rightarrow \TT\ps{X}, h\mapsto  \sum_{n \geq 0} h^{(n)} X^n \]
is a homomorphism of  $\CI$-algebras, and it can actually be
given by replacing the variable $t$ in the power series expansion for $h$ by $t+X$, expanding each $(t+X)^n$ using the binomial theorem, and rearranging to obtain a power series in $X$. 

Since elements of the Tate algebra may be evaluated at $z \in \CI$ such that $|z| \leq 1$, we can compose the map $\ts{}$ with evaluation at those elements, and we will write
$\ts{z}:\TT\to \CI\ps{X}$ for the $\CI$-algebra map
\[ h\mapsto \ts{z}(h)= \ts{}(h)|_{t = z}=\sum_{n=0}^\infty h^{(n)}(z) X^n .\]
Having in mind that $h^{(n)}$ corresponds to $\tfrac{1}{n!}\left(\tfrac{\partial}{\partial t}\right)^n(h)$ in characteristic zero, the map $\ts{z}$ is kind of a Taylor series expansion at $z$ (with $X$ standing for $(t-z)$), and the $h^{(n)}(z)$ are the Taylor coefficients.

Recognize that the map $\ts{z}:\TT\to \CI\ps{X}$ is still injective, as there is no element $h\in \TT\setminus \{0\}$ such that all hyperderivatives $h^{(n)}$ are divisible by $(t-z)$.

Furthermore, if $h(z) \neq 0$, then $\ts{z}(h)$ is invertible in $\CI\ps{X}$.

\medskip

In Section \ref{sec:link-to-motivic-group}, we need the following property of the map $\ts{\zeta}$, for $\zeta\in \Fqac$ a root of the prime polynomial $\pfrak$.

\begin{lem}\label{lem:D_zeta-isomorphism}
The composition $\ts{\zeta}\circ \chi_t:A\to \CI\ps{X}$ has image in $\Fq(\zeta)\ps{X}$ and extends to an isomorphism
\[ \ts{\zeta}\circ \chi_t:A_\pfrak \to \Fq(\zeta)\ps{X}. \]
\end{lem}

\begin{proof}
For all $a\in A$, $\ts{\zeta}(a(t))=a(\zeta+X)\in  \Fq(\zeta)[X]\subseteq \Fq(\zeta)\ps{X}\subseteq \CI\ps{X}$. In particular,
\begin{eqnarray*}
\ts{\zeta}(\chi_t(\theta)) &=& \zeta+X \quad \text{and} \\
\ts{\zeta}(\pfrak(t)) &=& \pfrak(\zeta)+\pfrak^{(1)}(\zeta)X+\ldots + \pfrak^{(d)}(\zeta)X^d \\
&=& \pfrak'(\zeta)X+\ldots + X^d.
\end{eqnarray*}
Since the image of $\pfrak$ lies in the maximal ideal $X\Fq(\zeta)\ps{X}$, the map $\ts{\zeta}\circ \chi_t$ extends to the $\pfrak$-adic completion $A_\pfrak$.
Even more,  this extension is an isomorphism
$\ts{\zeta}\circ \chi_t:A_\pfrak\to \Fq(\zeta)\ps{X}$, as the image of $\pfrak$ generates the maximal ideal, and also the residue fields are isomorphic.
\end{proof}

For later use, we will also extend the Frobenius twisting $\tau$ to $\CI\ps{X}$ and to $\TT\ps{X}$ linearly in $X$.
We observe the following commutation rules which are easily verified.

\begin{lem}\label{lem:Taylor-series-and-twisting}
\begin{enumerate}
\item For all $h\in \TT$:\hspace*{2.1cm} $(\ts{}(h))^\tau=\ts{}(h^\tau)$,
\item for all $h\in \TT$, and $z\in \CI$ with $|z|\leq 1$:\hspace*{3mm} $(\ts{z}(h))^\tau=\ts{z^q}(h^\tau)$,
\item for all $a\in A_\pfrak$, and $k=1,\ldots, d$:\hspace*{1.4cm}  $(\ts{\zeta_k}(a(t)))^\tau=\ts{\zeta_k^q}(a(t))$.
\end{enumerate}
\end{lem}

From $\ts{}$ and $\ts{z}$, we obtain the following family (in $n \geq 1$) of representations of $\CI$-algebras:
\begin{equation} \label{rhondef} \rho^{[n]} : \TT \rightarrow \TT^{n \times n}, \text{ defined by } h \mapsto \left( \begin{matrix} h & h^{(1)} & \cdots & h^{(n-1)} \\ 0 & h & \ddots & \vdots \\ \vdots & \ddots & \ddots & h^{(1)} \\ 0 & \cdots & 0 & h \end{matrix} \right),
\end{equation}
arising from the map $\ts{}$ by evaluation of $X$ at the obvious $n \times n$ nilpotent matrix, as well as
 $\rho_z^{[n]} : \TT \rightarrow \CI^{n \times n}$ for the $\CI$-algebra map
\begin{equation} \label{rhonevdef}
h \mapsto \rho^{[n]}(h)|_{t = z}. 
\end{equation}
Clearly, if $h(z) \neq 0$, then $\rho^{[n]}_z(h)$ is invertible.

\subsection{Anderson Generating Functions}\label{omegasection}

Throughout this article (except the last section), we will fix a Drinfeld module $\phi$ of rank $r$ defined over the extension $F$ of $K=\Fq(\theta)$. It is determined by 
\[\theta \mapsto \phi_\theta := \theta + g_1 \tau + \cdots + g_r \tau^r\] with $g_1,\ldots, g_r\in F$, $g_r \neq 0$. The exponential function of $\phi$ is denoted by $\exp_\phi$, and the period lattice $\Lambda_\phi=\ker(\exp_\phi)$ of $\phi$ is an $A$-lattice in $\CI$ of rank $r$. We also fix a basis $\{z_1,\ldots, z_r\}$ of this lattice.

The \textit{Anderson generating functions} corresponding to the $z_j$ are
\[ \omega_j(t) := \sum_{i \geq 0} \exp_\phi\Big(\frac{z_j}{\theta^{i+1}}\Big) t^i \in \TT,\]
and they generate the $\Fq[t]$-module of solutions to the equation
\[ Z^{\phi_\theta} = t \cdot Z \]
(see \cite[Rem.~6.3]{EP} or \cite[Thm.~3.6]{am:ptmsv}). One deduces that for all $a \in A$ and all $j=1,\ldots r$, one has
\begin{equation}  \label{omegacarlitztwisteq} \omega_j^{\phi_a} = a(t) \omega_j. 
\end{equation}

Now, the hyperderivatives commute with twisting. Thus, applying the $n$-th hyperderivative and using the Leibniz rule, we obtain
\begin{align} \label{omeganfnleqn} 
(\omega_j^{(n)})^{\phi_a} = (\omega_j^{\phi_a})^{(n)} = (a(t)\omega_j)^{(n)}=(a(t), a(t)^{(1)}, \dots, a(t)^{(n)} ) \cdot \left( \begin{matrix} \omega_j^{(n)} \\ \omega_j^{(n-1)} \\ \vdots \\ \omega_j  \end{matrix} \right ), \forall j = 1,\dots,r.  \end{align}
In fact, one may express \eqref{omeganfnleqn} more compactly with the representation $\rho^{[n+1]}$ of $\TT$ defined in \eqref{rhondef}.
From \eqref{omegacarlitztwisteq}, we obtain
\[ \phi_a \left( \rho^{[n+1]}(\omega_j)\right) =\rho^{[n+1]}(\omega_j^{\phi_a}) = \rho^{[n+1]}(a(t)\omega_j) = \rho^{[n+1]}(a(t))\rho^{[n+1]}(\omega_j), \forall j = 1,\dots,r.\]

\section{Explicit description of the coefficients of \boldmath$\omega_j^{(n)}(\zeta)$}
\label{sec:coefficients-are-torsion}

As above, let $\zeta \in \Fqac \subset \CI$ be a root of the monic irreducible polynomial $\pfrak \in A$ of degree $d$.
For $1 \leq j \leq r$ and $n \geq 0$, write
\begin{equation} \label{torsioncoeffseq}
\omega_j^{(n)}(\zeta) = \sum_{l = 0}^{d - 1} c_{j,(n),l} \zeta^{l}.
\end{equation}
for uniquely determined $c_{j,(n),l}\in \CI$ as in \eqref{TTrtof1eval}.

In \cite{MP},  it is shown that for the Anderson-Thakur function $\omega$, the corresponding coefficients are $\pfrak^{n+1}$-Carlitz torsion.
Here, we show the analogous result for the $c_{j,(n),l}$, and even stronger results.

\begin{prop}[{\bf Coefficients \boldmath$c_{j,(n),l}$ are torsion}] \label{prop:omega-n-integral} 
\ \\
For all $j\in \{1,\ldots, r\}$, $n\geq 0$ and $l\in \{ 0,1,\dots, d - 1\}$, the coefficients $c_{j,(n),l}$ of \eqref{torsioncoeffseq} are elements of the $\pfrak^{n+1}$-torsion $\phi[\pfrak^{n+1}]$.
\end{prop}

\begin{proof}
The proof is identical to the one in the Carlitz case.\\
From Lemma \ref{twistevallem} and \eqref{omeganfnleqn} above, for each root $\zeta_k$ ($k=1,\ldots d$) of $\pfrak$, we see that
\[\sum_{l = 0}^{d - 1} \phi_{\pfrak^{n+1}}(c_{j,(n),l}) \zeta_k^{\,l} = (\omega_j^{(n)})^{\phi_{\pfrak^{n+1}}}|_{t = \zeta_k} =\left(\pfrak(t)^{n+1}\omega_j\right)^{(n)}|_{t = \zeta_k} = 0,\]
since $\left(\pfrak(t)^{n+1}\omega_j\right)^{(n)}$ is divisible by $\pfrak(t)$.
By uniqueness of the coefficients, we conclude $\phi_{\pfrak^{n+1}}(c_{j,(n),l}) = 0$, for all $l = 0,1,\dots, d - 1$.
\end{proof}

The coefficients can be described even more explicitly. 
The key ingredients for their computation are Pellarin's identities
\begin{equation} \label{eq:omega-as-exp}
 \omega_j(t)= \sum_{m\geq 0} \frac{z_j^{q^m} e_m}{(\theta^{q^m}-t)},
\end{equation}
where $\exp_\phi = \sum_{m \geq 0} e_m \tau^m \in \CI\{\!\{ \tau \}\!\}$.
(This identity is obtained from the definition of $\omega_j$ by a change of the summation.)

\begin{defn}
For each $n \geq 0$, let
\[\qfrak_{(n)}(\theta,T):= \sum_{j=1}^d \frac{1}{\pfrak'(\zeta_j)}\frac{\pfrak(\theta)^{n+1}}{(\theta-\zeta_j)^{n+1}} \frac{\pfrak(T)}{(T-\zeta_j)} = \sum_{l= 0}^{d-1} \qfrak_{(n),l}(\theta)T^l \in A[T] \]
be the unique polynomial of degree strictly less than $d = \deg \pfrak$ in $T$ interpolating the map $\zeta_k \mapsto \frac{\pfrak(\theta)^{n+1}}{(\theta-\zeta_k)^{n+1}}$.
\end{defn}

\begin{rem}
Although the coefficients $\qfrak_{(n),l}$ of $\qfrak_{(n)}(\theta,T)$ are a priori in $A[\zeta]$, it follows from the invariance of the polynomial $\qfrak_{(n)}(\theta,T)$ under permutation of the $\zeta_k$ that the coefficients are indeed in $A$.

Notice further, that for all $k=1,\ldots d$:
\[ \qfrak_{(n)}(\theta,\zeta_k)= \frac{\pfrak(\theta)^{n+1}}{(\theta-\zeta_k)^{n+1}}=\frac{(\pfrak(\theta)-\pfrak(\zeta_k))^{n+1}}{(\theta-\zeta_k)^{n+1}}
=\left( \frac{\pfrak(\theta)-\pfrak(T)}{\theta-T}\right)^{n+1}|_{T=\zeta_k}. \]
Hence, $ \qfrak_{(n)}(\theta,T)$ and $\left( \frac{\pfrak(\theta)-\pfrak(T)}{\theta-T}\right)^{n+1}$ have to be congruent modulo $\pfrak(T)$. This implies that $\qfrak_{(n)}(\theta,T)$ equals the remainder of $\left( \frac{\pfrak(\theta)-\pfrak(T)}{\theta-T}\right)^{n+1}$ under division by $\pfrak(T)$.

For $n=0$, the polynomial $ \frac{\pfrak(\theta)-\pfrak(T)}{\theta-T}$ already is of degree $d-1$, and hence equals $\qfrak_{(0)}(\theta,T)$. The short computation in \cite{MP}, shows that 
\[ \qfrak_{(0),l}=\sum_{k=l+1}^d \alpha_k\theta^{k-l-1}\in A, \]
where we have written $\pfrak=\sum_{k=0}^d \alpha_k\theta^k\in A$ with $\alpha_0,\ldots,\alpha_d\in \Fq$.
\end{rem}

\begin{prop} \label{explicitcoeffsprop}
Fix $1 \leq j \leq r$ and $n \geq 0$. For each $l = 0, \dots, d-1$, we have
 \[ c_{j,(n),l}=\phi_{\qfrak_{(n),l}}\left( \exp_\phi(\frac{z_j}{\pfrak^{n+1}})\right)= \exp_\phi\left(\frac{z_j \qfrak_{(n),l}}{\pfrak^{n+1}}\right). \]
\end{prop}

\begin{proof}

The $c_{j,(n),l}$ are the coefficients of the unique polynomial $c_{j,(n)}(T)\in \CI[T]$ satisfying
$\deg_T(c_{j,(n)}(T)) < d$ and $c_{j,(n)}(\zeta_k)=\omega_j^{(n)}(\zeta_k)$, for all $1\leq k\leq d$. This gives
\[ c_{j,(n)}(T)=\sum_{i=1}^d \omega_j^{(n)}(\zeta_i) \prod_{k\neq i}\frac{T-\zeta_k}{\zeta_i-\zeta_k}
= \sum_{i=1}^d \frac{\omega_j^{(n)}(\zeta_i)}{\pfrak'(\zeta_i)}\frac{\pfrak(T)}{(T-\zeta_i)}. \]
By hyperdifferentiating \eqref{eq:omega-as-exp} (with respect to $t$), one has
\[  \omega_j^{(n)}(t)= \sum_{m\geq 0} \frac{z_j^{q^m} e_m}{ (\theta^{q^m}-t)^{n+1}}, \]
and hence, making this replacement we compute that
\[  c_{j,(n)}(T) = \sum_{m\geq 0} \frac{z_j^{q^m}e_m}{ \pfrak(\theta^{q^m})^{n+1}} \qfrak_{(n)}(\theta^{q^m},T). \qedhere \]
\end{proof}

\begin{thm} \label{Fqbasisthm}
Let $n\geq 0$. The set of coefficients 
\[ \{ c_{j,(m),l} \mid 1\leq j\leq r; 0\leq m\leq n; 0\leq l\leq d-1 \} \]
forms an $\Fq$-basis of $\phi[\pfrak^{n+1}]$.   
\end{thm}

\begin{proof}
Inductively, it is enough to show that for any $n\geq 0$ the set $\{ c_{j,(n),l} \mid 1\leq j\leq r; 0\leq l\leq d-1 \}$ is an $\Fq$-basis of a complement of $\phi[\pfrak^{n}]$ inside 
$\phi[\pfrak^{n+1}]$. One already knows that
\[  \left\{ \exp_\phi\left(\frac{z_ja_k}{\pfrak^{n+1}}\right) \mid j=1,\ldots, r; k=1,\ldots, d \right\} \]
is such a basis, whenever the $\Fq$-span of $a_1,\ldots, a_d\in A$ surjects onto $A/\pfrak$ via the residue map $A\to A/\pfrak$.
Since $c_{j,(n),l}=\exp_\phi\left(\frac{z_j \qfrak_{(n),l}}{\pfrak^{n+1}}\right)$, we have to show that the $\Fq$-span of 
$\qfrak_{(n),0},\ldots ,\qfrak_{(n),d-1}\in A$ surjects onto $A/\pfrak$.

By extension of constants, this is the same as that the $\Fq(\zeta)$-span of the elements
$\qfrak_{(n),0},\ldots ,\qfrak_{(n),d-1}\in A(\zeta)=\Fq(\zeta)[\theta]$ surjects onto 
$\Fq(\zeta)[\theta]/\pfrak$. By the Chinese remainder theorem, the latter is isomorphic
$\Fq(\zeta)^d$ via $\bar{\theta}\mapsto (\zeta_1,\ldots, \zeta_d)$.

On the other hand, we have
\[ \begin{pmatrix} 1 & \zeta_1 & \zeta_1^2 & \cdots & \zeta_1^{d - 1} \\ 1 & \zeta_2 & \zeta_2^2 & \cdots & \zeta_2^{d - 1} \\ \vdots \\ 1 & \zeta_d & \zeta_d^2 & \cdots & \zeta_d^{d - 1} \end{pmatrix}  \begin{pmatrix} \qfrak_{(n),0}\\ \qfrak_{(n),1}\\ \vdots \\ \qfrak_{(n),d-1} \end{pmatrix} 
= \begin{pmatrix} \qfrak_{(n)}(\theta,\zeta_1) \\  \qfrak_{(n)}(\theta,\zeta_2)\\ \vdots \\  \qfrak_{(n)}(\theta,\zeta_d) \end{pmatrix}
= \begin{pmatrix} \frac{\pfrak(\theta)^{n+1}}{(\theta-\zeta_1)^{n+1}}\\
\frac{\pfrak(\theta)^{n+1}}{(\theta-\zeta_2)^{n+1}} \\ \vdots \\  
\frac{\pfrak(\theta)^{n+1}}{(\theta-\zeta_d)^{n+1}} \end{pmatrix}.
\]
So the  $\Fq(\zeta)$-span of 
$\qfrak_{(n),0},\ldots ,\qfrak_{(n),d-1}$ is the same as that of
$\frac{\pfrak(\theta)^{n+1}}{(\theta-\zeta_1)^{n+1}},\ldots, \frac{\pfrak(\theta)^{n+1}}{(\theta-\zeta_d)^{n+1}}$.
Finally, via the projection $\Fq(\zeta)[\theta]\to \Fq(\zeta)^d$ above, $\frac{\pfrak(\theta)^{n+1}}{(\theta-\zeta_k)^{n+1}}$ is mapped to the vector
$(0,\ldots, 0, \pfrak'(\zeta_k)^{n+1},0,\ldots, 0)$, i.e.~a nonzero multiple of the $k$-th
standard basis vector. Hence, the  $\Fq(\zeta)$-span of these elements surjects onto
$\Fq(\zeta)^d$.
\end{proof}

As a corollary, we get the non-vanishing of these coefficients.

\begin{cor}[{\bf Non-vanishing of torsion coefficients of \boldmath$\omega_j^{(n)}(\zeta)$}] \label{nonvantorsioncoeffcor}
Let $c_{j,(n),l}$ be as defined in \eqref{torsioncoeffseq} for $\omega_j^{(n)}(\zeta)$, for $n\geq 0$, $\,j=1,\ldots, r$, and $l = 0, 1,\dots, d-1$. We have
\[ c_{j,(n),l} \in \phi[\pfrak^{n+1}] \setminus \phi[\pfrak^{n}], \quad \forall l = 0,1,\dots,d-1. \]
In particular, these coefficients are non-zero for all $n \geq 0$, all $j=1,\ldots, r$, and all $l = 0,1,\dots,d-1$.
\end{cor}

Furthermore, we obtain that the values $\omega_j^{(n)}(\zeta_k)$ generate the $\pfrak^{n+1}$-torsion extension above $F(\zeta)$.

\begin{cor}\label{cor:generation-by-omega-n-zeta-k}
For all $n\geq 0$,
\[  F(\zeta, \omega_j^{(n)}(\zeta_k)| 1\leq j\leq r;\, 1\leq k\leq d) = F_n(\zeta)
= F(\zeta,\phi[\pfrak^{n+1}]), \]
where $F_n:=F(\phi[\pfrak^{n+1}])$ denotes the $\pfrak^{n+1}$-torsion extension of $F$.
\end{cor}

\begin{rem}
In the Carlitz case, the theorem of Angl\`es-Pellarin was even stronger: They only needed the evaluation of the Anderson-Thakur function (resp.~its $n$-th hyperderivative) at one of the $\zeta_k$ for generating the field extension $F_n(\zeta)/F(\zeta)$. We conjecture that also in the general case, we have
\[   F_n(\zeta)=F(\zeta, \omega_j^{(n)}(\zeta)| 1\leq j\leq r). \]
However, a proof would be much more complicated than in the Carlitz case.
\end{rem}

\begin{rem}
If the Drinfeld module $\phi$ is defined over the ring of integers $O_F$ of $F$, and has everywhere good reduction (i.e.~the coefficients of $\phi_\theta$ lie in $O_F$, and the highest coefficient is  in $O_F^\times$), then all elements in $\phi[\pfrak^{n+1}]$ are integral over $O_F$, and hence also all $\omega_j^{(n)}(\zeta_k)$ are integral elements. 
\end{rem}

\section{The Galois Action}\label{sec:galois-action}

We still assume that the Drinfeld module $\phi$ has rank $r$ and is defined over the extension $F$ of $K$, with associated lattice 
\[z_1A + \cdots +z_r A \subset \CI.\] 
We have seen that for all $0 \leq m \leq n$, $1 \leq j \leq r$, and $1\leq k\leq d$ the elements $\omega_j^{(m)}(\zeta_k)$ lie in the extension $F(\zeta)(\phi[\pfrak^{n+1}])$ which is a Galois extension of $F$. The next result shows that there is a very clean description of the action of Galois on the vector
\[ X_\zeta^{[n+1]} := (\omega_1^{(n)}, \dots, \omega_1^{(1)},\omega_1,\dots,\omega_r^{(n)}, \dots, \omega_r^{(1)},\omega_r)_{t = \zeta}^{tr} \]

For this, we employ the Galois representation on the $\pfrak$-adic Tate-module $T_\pfrak(\phi)$ with respect to the basis given by the sequences $(\exp_\phi(z_1/\pfrak^{n+1}))_{n\geq 0}$ up to\linebreak $(\exp_\phi(z_r/\pfrak^{n+1}))_{n\geq 0}$. Hence, this representation 
\begin{equation}\label{eq:galois-rep}
 \varphi_\pfrak: \Gal(F^{\rm sep}/F) \twoheadrightarrow \Gal(F(\phi[\pfrak^{\infty}])/F) \to \GL_r(A_\pfrak), \sigma \mapsto \mathcal{A}_\sigma 
\end{equation}
is given such that
\[ \sigma(\exp_\phi(z_j/\pfrak^{n+1})) = \exp_{\phi}(\sum_{k = 1}^r a_{\sigma,j,k} z_k / \pfrak^{n+1}), \]
where $a_{\sigma,j,k}$ is the $(j,k)$-th coefficient of $\mathcal{A}_\sigma$ (modulo $\pfrak^{n+1}$).

\begin{lem}[{\bf Galois action on \boldmath$\omega_j^{(n)}(\zeta)$}] \label{galoisactionlem} \

Let $\sigma\in \Gal(F^{\rm sep}/F)$, and $\mathcal{A}_\sigma=\varphi_\pfrak(\sigma)$ as above
 with entries $a_{\sigma,j,i}\in A_\pfrak$. Further, let 
 \[ X_\zeta^{[n+1]} := (\omega_1^{(n)}, \dots, \omega_1^{(1)},\omega_1,\dots,\omega_r^{(n)}, \dots, \omega_r^{(1)},\omega_r)_{t = \zeta}^{tr}. \]

 Then in block matrix form
\[ \sigma\left( X_{\zeta}^{[n+1]}\right) = \Bigl(\rho_{\sigma(\zeta)}^{[n+1]}\bigl(a_{\sigma,j,i}\bigr)\Bigr)_{1 \leq j,i \leq r}X_{\sigma(\zeta)}^{[n+1]}. \]

\end{lem}
\begin{proof}

Using the explicit coefficients obtained in Proposition \ref{explicitcoeffsprop}, for any $0 \leq m \leq n$, as well as Lemma \ref{twistevallem}(2) and Equation \eqref{omeganfnleqn}, we obtain
\begin{eqnarray*}
\sigma(\omega_j^{(m)}(\zeta)) &=& \sum_{i = 0}^{d-1} \exp_\phi\left(\frac{\sum_{k = 1}^r a_{\sigma,j,k} z_k \qfrak_{(m),i}}{\pfrak^{m+1}}\right) \sigma(\zeta)^i\\
&=& \sum_{k = 1}^r  \sum_{i = 0}^{d-1} \phi_{a_{\sigma,j,k}}
\left(\exp_\phi\Bigl( \frac{z_k \qfrak_{(m),i}}{\pfrak^{m+1}}\Bigr) \right) \sigma(\zeta)^i\\
& =& \sum_{k = 1}^r (\omega_k^{(m)})^{\phi_{a_{\sigma,j,k}}}|_{t = \sigma(\zeta)}= \sum_{k = 1}^r (\chi_t(a_{\sigma,j,k})\omega_k)^{(m)}(\sigma(\zeta)). 
\end{eqnarray*} 
 
Putting this in matrix form, we deduce the claim. 
\end{proof}

We can say the previous thing even cleaner. For this, we consider the ring homomorphisms
\[  \ts{\zeta}:\TT\to  \CI\ps{X}, h\mapsto  \mathcal{D}(h)|_{t=\zeta}=  \sum_{n=0}^\infty h^{(n)}(\zeta)X^n \]
defined earlier. We also consider $a\mapsto \ts{\zeta}(\chi_t(a))=\ts{\zeta}(a(t))$ for $a\in A$, as well as its $\pfrak$-adically continuous 
extension to $A_\pfrak$ which induces an isomorphism to $\Fq(\zeta)\ps{X}$ as shown in Lemma \ref{lem:D_zeta-isomorphism}. However, we will still write the image of the composition as $\ts{\zeta}(a)$.
We let the Galois group $\Gal(F^{\rm sep}/F)$ act on these power series by acting on the coefficients -- as long as the coefficients lie in $F^{\rm sep}$.

\pagebreak

\begin{prop}\label{galoisactionprop}
 For each $\sigma \in \Gal(F^{\rm sep}/F)$, let $\mathcal{A}_\sigma=\varphi_\pfrak(\sigma) \in \GL_r(A_\pfrak)$. Then 
\[  \sigma \bigl( \ts{\zeta} \begin{pmatrix}
\omega_1 \\ \vdots \\ \omega_r \end{pmatrix} \bigr) =
\ts{\sigma(\zeta)}(\mathcal{A}_\sigma)\cdot 
\ts{\sigma(\zeta)} \begin{pmatrix}
\omega_1 \\ \vdots \\ \omega_r \end{pmatrix}\!. \,
  \footnote{As usual, we apply the maps $\ts{\zeta}$ and $\ts{\sigma(\zeta)}$ entry-wise.}
 \]
\end{prop}

\begin{proof}
This is just the limit version of the previous lemma written in a different form.
\end{proof}

\begin{rem}\label{rem:torsion-and-zeta}
\begin{enumerate}
\item If $F(\phi[\pfrak^{\infty}])$ and $F(\zeta)$ are linearly disjoint over $F$, the Galois group $\Gal(F(\phi[\pfrak^{\infty}])/F)$ is isomorphic to $\Gal(F(\zeta,\phi[\pfrak^{\infty}])/F(\zeta))$, and hence
the subgroup $\Gal(F^{\rm sep}/F(\zeta))$ of $\Gal(F^{\rm sep}/F)$ still surjects onto the group $\Gal(F(\phi[\pfrak^{\infty}])/F)$. The previous proposition, therefore, gives a uniform description of the Galois action of $\Gal(F(\phi[\pfrak^{\infty}])/F)$ in terms of the Anderson generating functions, but without a twist in the $\zeta$.
\item If the completion $F_\infty$ of $F$ with respect to $|\cdot |$ is a finite extension of $K_\infty$ -- in particular, if $F$ is finite over $K$ -- the condition that the extensions $F(\phi[\pfrak^{\infty}])$ and $F(\zeta)$ are linearly disjoint over $F$, holds for infinitely many primes $\pfrak$.
Namely, $F(\phi[\pfrak^{\infty}])$ is contained in $F_\infty(\phi[\pfrak^{\infty}])$ which equals
$F_\infty(z_1,\ldots, z_r)$ by \cite[Prop.~1.2]{Gek}, and is a finite extension of $F_\infty$. Therefore, the algebraic closure of $\Fq$ in $F_\infty(z_1,\ldots, z_r)$ is a finite extension $\FF$ of $\Fq$. Therefore, for all primes $\pfrak$ for which $\deg(\pfrak)$ is prime to 
$[\FF:\Fq]$, the extensions $F(\phi[\pfrak^{\infty}])$ and $F(\zeta)$ are linearly disjoint over $F$.
\item Using a theorem of Pink and R\"utsche on the adelic openness of the Galois image in the adelic points of the Mumford-Tate group (see \cite[Thm.~0.2]{PR}), one can even show that there are only finitely many primes $\pfrak$ for which the extensions $F(\phi[\pfrak^{\infty}])$ and $F(\zeta)$ are not linearly disjoint. Since, we don't need it, we don't go into details here.
\end{enumerate}
\end{rem}

\section{Direct link to the motivic Galois group}\label{sec:link-to-motivic-group}

In this section, we show explicitly that the image of the Galois representation $\varphi_\pfrak$ lies in the $A_\pfrak$-points of the 
motivic Galois group. This part is a generalization to arbitrary primes $\pfrak$ of a result of Chang and Papanikolas (see \cite[Cor.~3.2.4 and Thm.~3.5.1]{cc-mp:aipldm}). 
Since the dual $t$-motive is only defined over perfect base fields, we further need $\Fper:=F^{1/p^\infty}\subseteq \CI$, the perfect closure of $F$. We will also use $\Fac\subseteq \CI$ the algebraic closure of $F$, and as before $\Fsep$ the separable algebraic closure of $F$. We remind the reader that $\Fac$ is the composite of $\Fper$ and $\Fsep$ (see \cite[Prop.~6.11]{Lang}), and that
one has the natural isomorphism of Galois groups 
$\Gal(\Fsep/F)\cong \Gal(\Fac/\Fper)$ which we will use in the following without mentioning.

We consider the $t$-motive $\dumot$ associated to $\phi$ in the same way as in \cite[\S 3.3]{cc-mp:aipldm}, but with $\Fac$ replaced by $\Fper$. That means that $\dumot$ is a $\Fper(t)[\tau^{-1}]$-module, it is isomorphic to $\left(\Fper(t)\right)^r$ as $\Fper(t)$-module, and multiplication by $\tau^{-1}$ on $\dumot$ with respect to the standard basis is represented by the matrix 
\[ \Phi_\phi = \begin{pmatrix} 0 & 1 & \cdots & 0 \\ 
\vdots & \vdots & \ddots & \vdots \\
0 & 0 & \cdots & 1 \\
\frac{(t-\theta)}{g_r^{q^{-r}}} & -\frac{g_1^{q^{-1}}}{g_r^{q^{-r}}} & \cdots & -\frac{g_{r-1}^{q^{-r+1}}}{g_r^{q^{-r}}} 
\end{pmatrix}, \]
where as before $\phi_\theta := \theta + g_1 \tau + \cdots + g_r \tau^r\in F\{\tau\}$.
As in \cite[\S 3.4]{cc-mp:aipldm}, we define
\[  \Upsilon:= \begin{pmatrix} \omega_1 & \omega_1^{\tau} &\ldots & \omega_1^{\tau^{r-1}} \\
\vdots & & &  \vdots \\
  \omega_r & \omega_r^{\tau} &\ldots & \omega_r^{\tau^{r-1}}
\end{pmatrix} \in \GL_r(\TT). \]
Then, for a  suitable matrix $V\in \GL_r(\Fper)$ (whose explicit shape will not be needed in the sequel), the matrix
\[ \Psi_\phi:= V^{-1} (\Upsilon^\tau)^{-1}\in \GL_r(\TT).\]
is a rigid analytic trivialization of $\dumot$ for $\Phi_\phi$.

From Proposition \ref{galoisactionprop}, we deduce the following analog of \cite[Cor.~3.2.4]{cc-mp:aipldm}.

\begin{prop}\label{prop:sigma-on-psi}
For each $\sigma\in \Gal(F^{\rm sep}/F)$, one has
\[  \sigma \bigl( \ts{\zeta}(\Upsilon^\tau) \bigr) = \ts{\sigma(\zeta)}( \mathcal{A}_\sigma) \cdot\ts{\sigma(\zeta)}(\Upsilon^\tau),   \]
as well as
\[ \sigma \bigl(\ts{\zeta}(\Psi_\phi) \bigr) =\ts{\sigma(\zeta)}(\Psi_\phi) \cdot \ts{\sigma(\zeta)}( \mathcal{A}_\sigma)^{-1} ,\]
where $\mathcal{A}_\sigma=\varphi_\pfrak(\sigma)\in \GL_r(A_\pfrak)$.\\
\end{prop}

\begin{proof}
Twisting commutes with the Galois action on the coefficients.
Therefore, using Lemma \ref{lem:Taylor-series-and-twisting} and  Proposition~\ref{galoisactionprop}, one obtains for the $i$-th column of $\Upsilon^\tau$:
\begin{eqnarray*}
  \sigma \bigl( \ts{\zeta}\left( \begin{smatrix}
\omega_1 \\ \vdots \\ \omega_r \end{smatrix}^{\!\!\tau^i}\right) \bigr) 
&=& \left(   \sigma \bigl( \ts{\zeta^{1/q^i}} \begin{smatrix}
\omega_1 \\ \vdots \\ \omega_r \end{smatrix}\bigr)   \right)^{\!\!\tau^i}
= \left( \ts{\sigma(\zeta^{1/q^i})}(\mathcal{A}_\sigma)\cdot 
\ts{\sigma(\zeta^{1/q^i})} \begin{smatrix}
\omega_1 \\ \vdots \\ \omega_r \end{smatrix} \right)^{\!\!\tau^i}\\
&=&  \ts{\sigma(\zeta)}( \mathcal{A}_\sigma) \cdot  \ts{\sigma(\zeta)}\left( \begin{smatrix}
\omega_1 \\ \vdots \\ \omega_r \end{smatrix}^{\!\!\tau^i}\right).
\end{eqnarray*}

Using $ \Psi_\phi= V^{-1} (\Upsilon^\tau)^{-1}$, and taking into account that $V\in \GL_r(\Fper)$ is fixed by $\sigma$, we finally get
\begin{eqnarray*}
\sigma \bigl( \ts{\zeta}(\Psi_\phi)\bigr) &=& \sigma \bigl( \ts{\zeta}\left( V^{-1} (\Upsilon^\tau)^{-1}\right) \bigr)
= V^{-1}\cdot \left( \sigma \bigl( \ts{\zeta}(\Upsilon^\tau) \bigr) \right)^{-1} \\
&=& V^{-1}\cdot \ts{\sigma(\zeta)}(\Upsilon^\tau)^{-1} \cdot \ts{\sigma(\zeta)}( \mathcal{A}_\sigma)^{-1} \\
&=& \ts{\sigma(\zeta)}(\Psi_\phi) \cdot \ts{\sigma(\zeta)}( \mathcal{A}_\sigma)^{-1}.
\end{eqnarray*}
\end{proof}

We use the same characterization of the Galois group $\Gamma_{\Psi_\phi}$ associated to $\Psi_\phi$ as in \cite[Sect.~4.2]{mp:tdadmaicl}. Although the hypotheses there are stronger than in our setting, more recent results in difference Galois theory provide the necessary relaxed hypotheses. We refer here to the work of A.~Bachmayr \cite[Sect.~2]{am:dvnt} and \cite[Ch.~1]{am:phd}.\\
The fact that the motivic Galois group $\Gamma_{M_\phi}$, i.e.~the Tannakian Galois group of the $t$-motive $M_\phi$, is isomorphic to the Galois group $\Gamma_{\Psi_\phi}$ associated to $\Psi_\phi$ is guaranteed by the general framework in \cite[Corollary 7.8]{am:capvt}.

First of all, let $\LL=\Frac(\TT)$ be the field of fractions of $\TT$. We denote by
$P = \Fper(t)[(\Psi_\phi)_{ij}, \det(\Psi_\phi)^{-1}]\subseteq \LL$ the so-called Picard-Vessiot ring for $\Phi_\phi$.
It is isomorphic to the quotient of $\Fper(t)[Y_{ij},\det(Y)^{-1}]$ by the ideal
\[ S:=\{ h\in \Fper(t)[Y_{ij},\det(Y)^{-1}] \mid h(\Psi_\phi)=0 \}.\]
Then, for each $\Fq(t)$-algebra $R$, the $R$-points of $\Gamma_{\Psi_\phi}$ are defined to be the difference automorphisms of the difference ring $(P \otimes_{\Fq(t)} R, \tau^{-1}\otimes \id_R)$ \footnote{This means that the endomorphism on this tensor product is given by $\tau^{-1}$ on the first component and by the identity on the second.}. Explicitly, one has
\[ \Gamma_{\Psi_\phi}(R)=\left\{ \mathcal{A}\in \GL_r(R) \mid  \forall h\in S: h( \Psi_\phi\otimes \mathcal{A})=0 \in P\otimes R \right\}. \]

Here and in the following, the tensor product $\Psi_\phi\otimes \mathcal{A}$ of two square matrices of size $r$ denotes the square matrix of the same size with entries
$(\Psi_\phi\otimes \mathcal{A})_{ij} = \sum_{k=1}^r (\Psi_\phi)_{ik}\otimes (\mathcal{A})_{kj}$.

\begin{thm}\label{thm:image-of-gal-rep-in-motivic-gal-grp}
Let $\bA_{\pfrak(t)}$ be the completion of $\Fq[t]$ at the prime ideal $\pfrak(t)\in \Fq[t]$, and 
$\KK_{\pfrak(t)}$ its field of fractions. Consider
 $\chi_t$ extended continuously to a map $\chi_t:A_\pfrak\to \bA_{\pfrak(t)}\subseteq \KK_{\pfrak(t)}$. Then
 \[\chi_t\bigl( \varphi_\pfrak(\Gal(F^{\rm sep}/F))\bigr) \subseteq 
 \Gamma_{\Psi_\phi}(\bA_{\pfrak(t)}):= \GL_r(\bA_{\pfrak(t)})\cap \Gamma_{\Psi_\phi}(\KK_{\pfrak(t)}). \]
\end{thm}

For the proof, we will consider the Taylor series homomorphism $\ts{\zeta}$ on the tensor product $\TT\otimes_{\Fq[t]} \bA_{\pfrak(t)}$
as a homomorphism (by abuse of notation also denoted by $\ts{\zeta}$)
\[ \ts{\zeta}: \TT\otimes_{\Fq[t]} \bA_{\pfrak(t)}\to \CI\ps{X}, \sum_i h_i\otimes a_i\mapsto \sum_i  \ts{\zeta}(h_i)\cdot \ts{\zeta}(a_i). \]
This homomorphism $\ts{\zeta}$ might not be injective in general, but we have the following lemma.

\begin{lem}\label{lem:Dzeta-on-tensor-product}
Let as before $\zeta_1,\ldots,\zeta_d$ be all roots of $\pfrak$.
Let $f\in \TT\otimes_{\Fq[t]} \bA_{\pfrak(t)}$ be such that $\ts{\zeta_k}(f)=0$ for all $k=1,\ldots, d$. Then $f=0$.
\end{lem}

\begin{proof}
As $\ts{\zeta_k}$ is injective on $\TT$, we can extend $\ts{\zeta_k}$ to a ring homomorphism 
on the localization $\LL\otimes_{\TT}\left( \TT\otimes_{\Fq[t]} \bA_{\pfrak(t)}\right) \cong \LL\otimes_{\Fq[t]} \bA_{\pfrak(t)} \cong 
\LL\otimes_{\Fq(t)} \KK_{\pfrak(t)}$,
\[ \widetilde{\ts{\zeta_k}}: \LL\otimes_{\Fq(t)} \KK_{\pfrak(t)} \to \CI\ls{X}. \]
The aim is to show that the set
\[ I := \{ g \in \LL\otimes_{\Fq(t)} \KK_{\pfrak(t)} \mid \forall k=1,\ldots, d: \widetilde{\ts{\zeta_k}}(g)=0 \} \]
equals $\{0\}$.
This set $I$ is an ideal, since all $\widetilde{\ts{\zeta_k}}$ are ring homomorphisms. It is even stable under $\tau\otimes \id_{\KK_{\pfrak(t)}}$ by Lemma \ref{lem:Taylor-series-and-twisting}. Hence, it is a difference ideal of the difference ring 
$(\LL\otimes_{\Fq(t)} \KK_{\pfrak(t)}, \tau\otimes \id_{\KK_{\pfrak(t)}})$.
However, $\LL$ is a field with $\tau$-invariants $\Fq(t)$ (see \cite[Lemma 3.3.2]{mp:tdadmaicl}), and hence the difference ideal $I$ is generated by elements in $\KK_{\pfrak(t)}$ (see \cite[Lemma 1.2.9]{am:phd}). But $\widetilde{\ts{\zeta_k}}|_{\KK_{\pfrak(t)}}$ is injective, and therefore $I=0$.
\end{proof}

\begin{proof}[Proof of Thm.~\ref{thm:image-of-gal-rep-in-motivic-gal-grp}]
Let $\sigma\in \Gal(F^{\rm sep}/F)$, and $\mathcal{A}_\sigma=\varphi_\pfrak(\sigma)$. We will prove the theorem by showing that for arbitrary $h\in S$, the condition $h(\Psi_\phi \otimes \chi_t(\mathcal{A}_\sigma^{-1}))=0$ is satisfied. 
Of course, it is sufficient to consider those $h\in S$ whose coefficients (with respect to the $Y_{ij}$) lie in $\Fper[t]$. So let  $h\in S$ be such an element, and denote by $h_\zeta\in \CI\ps{X}[Y_{ij},\det(Y)^{-1}]$ its image when mapping the coefficients via $\ts{\zeta}$. Then
\begin{eqnarray*}
\ts{\zeta}\bigl( h(\Psi_\phi\otimes \chi_t(\mathcal{A}_\sigma^{-1})) \bigr) 
&=& h_{\zeta}\left( \ts{\zeta}(\Psi_\phi)\cdot \ts{\zeta}(\mathcal{A}_\sigma^{-1})\right)\\
&=& h_{\zeta}\left( \sigma \bigl( \ts{\sigma^{-1}(\zeta)}(\Psi_\phi)\bigr) \right)\quad \text{by Prop.~\ref{prop:sigma-on-psi}}\\
&=& \sigma \left( h_{\sigma^{-1}(\zeta)}(  \ts{\sigma^{-1}(\zeta)}(\Psi_\phi) )\right) \quad \text{as coeff.~of $h$ are in }\Fper[t] \\
&=& \sigma  \ts{\sigma^{-1}(\zeta)}\left(  h(\Psi_\phi)\right) \quad =\,\, 0.
\end{eqnarray*}
As this holds for any root $\zeta$ of $\pfrak$, the claim follows from Lemma \ref{lem:Dzeta-on-tensor-product}.
\end{proof}

\section{Variation of the lattice} \label{sec:modular-forms}

In this final section, we briefly address a connection to modular functions and vector valued modular forms. Namely, we consider the values $\omega_j^{(n)}(\zeta)$ as functions of the basis $\bm{z}=(z_1, \ldots, z_r)$ of the lattice and let $\bm{z}$ vary.

More precisely, let $\Omega^r$ denote Drinfeld's period domain (also called Drinfeld symmetric space), i.e.~the complement in $\mathbb{P}_{\CI}^{r-1}$ of the $K_\infty$-rational hyperplanes. The $\CC_\infty$ points of this space will be viewed as vectors $\bm{z} = (z_1,\dots,z_r) \in \CC_\infty^r \setminus \{K_\infty-hyperplanes\}$ whose last coordinate $z_r$ equals $\pitilde$, a fixed fundamental period of the Carlitz module.

For $\bm{z}\in \Omega^r$, let $\phi_{\bm{z}}$ be the Drinfeld module associated to the lattice $z_1A+\ldots +z_rA$, and as before, let $\zeta \in \CC_\infty$ be a root of the monic irreducible polynomial $\pfrak \in A$ of degree $d$.

\begin{defn}
For each $j = 1,2,\dots,r$, and $n\geq 0$, let $\underline{\omega}_j^{(n)}(\zeta) : \Omega^r \rightarrow \CC_\infty$ be the function defined by
\begin{equation} \label{modularAGFdef}
 \bm{z} \mapsto \sum_{k = 0}^{d-1} \exp_{\phi_{\bm{z}}}\left(\frac{\bm{z}\bm{e}_j \qfrak_{(n),k}}{\pfrak^{n+1}}\right)\zeta^k, 
\end{equation}
where $\bm{e}_j$ is the $j$-th column of the $r\times r$ identity matrix.
\end{defn}
Clearly by Prop.~\ref{explicitcoeffsprop}, for fixed $\bm{z} \in \Omega^r$, the values $\underline{\omega}_j^{(n)}(\zeta)(\bm{z})$ are exactly the previously defined evaluations $\omega_j^{(n)}(\zeta)$. 

It follows trivially from the definition \eqref{modularAGFdef} that the functions $\underline{\omega}_j^{(n)}(\zeta)$ are modular functions of weight $-1$ for the principal congruence subgroup of level $\pfrak^{n+1}$
\[ \Gamma(\pfrak^{n+1}) := \{\gamma \in \GL_r(A) : \gamma\equiv Id \pmod{\pfrak^{n+1}}\}, \]
due to the connection of universal torsion with Eisenstein series of weight $1$. 

Furthermore, one can deduce from Theorem \ref{Fqbasisthm} and \cite[Proposition 2.6 (i)]{GekelerIV} that the field generated over $\CC_\infty$ by the functions $\omega_j^{(m)}(\zeta)$ (with $0\leq m\leq n$ as in Theorem \ref{Fqbasisthm}) equals Gekeler's function field $\widetilde{\mathcal{F}}_r(\pfrak^{n+1})$. So, the specializations at roots of unity of Anderson generating functions do much more than just generate torsion extensions, they rigid analytically interpolate generators for function fields of Drinfeld modular curves of prime power levels!

\bigskip

In \cite{Perk}, the second author considered the Anderson generating functions themselves as functions on Drinfeld's period domain $\Omega^r$ with values in the Tate algebra $\TT$, and obtained that they are vector valued modular forms.

The case $r=2$ has already been encountered by Pellarin in \cite{FPannals}, and further studied by Pellarin and the second author in \cite{PP}. In this case, there is even a close connection of the Anderson generating functions to vector valued Eisenstein series.

It is desirable to have such a close connection also in the higher rank case. This will be in the scope of future research.


\end{document}